\documentclass{amsart}

\usepackage{amssymb}
\usepackage{graphicx}
\usepackage{MnSymbol}

\usepackage{enumerate}

\usepackage{amsmath,amscd}

\usepackage{amsthm}

\title{On $\kappa$-existentially closed groups with few automorphisms}

\author{Samuel M. Corson}
\address{E. T. S. I. I. Universidad Polit\'{e}cnica de Madrid, Jos\'{e} Guti\'{e}rrez Abascal 2, 28006 Madrid, Spain}
\email{sammyc973@gmail.com}
\author{Saharon Shelah}
\address{Einstein Institute of Mathematics, The Hebrew University of Jerusalem, Jerusalem 91904 Israel}
\address{Department of Mathematics, Rutgers University, Piscataway, NJ 08854 USA}
\email{shelah@math.huji.ac.il}

\theoremstyle{definition}\newtheorem{theorem}{Theorem}
\theoremstyle{definition}
\theoremstyle{definition}

\theoremstyle{definition}

\theoremstyle{definition}

\theoremstyle{definition}
\theoremstyle{definition}

\theoremstyle{definition}

\theoremstyle{definition}\newtheorem{corollary}[theorem]{Corollary}
\theoremstyle{definition}
\theoremstyle{definition}\newtheorem{definition}[theorem]{Definition}
\theoremstyle{definition}
\theoremstyle{definition}
\theoremstyle{definition}
\theoremstyle{definition}
\theoremstyle{definition}\newtheorem{lemma}[theorem]{Lemma}
\theoremstyle{definition}
\theoremstyle{definition}
\theoremstyle{definition}\newtheorem{claim}[theorem]{Claim}
\theoremstyle{definition}
\theoremstyle{definition}
\theoremstyle{definition}\newtheorem{notation}[theorem]{Notation}

\newcommand{\Aut}{\operatorname{Aut}}

\newcommand{\cf}{\operatorname{cf}}
\newcommand{\modvar}{\operatorname{mod}}

\newcommand{\Len}{\operatorname{Len}}

\begin{document}



\thanks{The first author is partially supported by RYC2023-045493-I.  The work of the second author is supported by ISF 2320/23: The Israel Science Foundation (ISF) (2023-2027).  Paper number 1280 on Shelah's archive.}

\begin{abstract}
We give constructions of $\kappa$-existentially closed groups having few automorphisms, giving negative solutions to a question of Kaya and Kuzucuo\u{g}lu in the Kourovka Notebook.

\end{abstract}

\maketitle

\begin{section}{Introduction}

Recall that a group $G$ is \emph{$\kappa$-existentially closed} if every system $\Sigma$ of equations and inequations with coefficients in $G$, with $|\Sigma| < \kappa$, having a solution in a larger group $H \geq G$ has a solution in $G$ itself.  This notion has also been called \emph{$\kappa$-algebraically closed}, or \emph{algebraically closed groups} when $\kappa = \aleph_0$  \cite{Scott}, \cite[\textsection IV.8]{LS}.

Kaya and Kuzucuo\u{g}lu have asked the following (see \cite[Problem 20.40]{KhMaz}).

\vspace{.25cm}

\noindent \textbf{Question.}  Let $G$ be a $\kappa$-existentially closed group of cardinality $\lambda > \kappa$, where $\kappa$ is a regular cardinal.  Is it true that $|\operatorname{Aut}(G)| = 2^{\lambda}$?

\vspace{.25cm}

\noindent It is known \cite{KayaKuzucuogluLongobardiMaj} that a $\kappa$-existentially closed group of regular cardinality $\kappa$ will have $|\operatorname{Aut}(G)| = 2^{\kappa}$, and so the question deals with the remaining case.  We show that the answer to the question is often negative.

\begin{theorem}\label{powercardinal}  Suppose that $\aleph_1 \leq \cf(\kappa) = \kappa < \mu$ and $\lambda = 2^{\mu}$.  Then there exists a $\kappa$-existentially closed group $G$ of cardinality $\lambda$ with $|\Aut(G)| = \lambda$.
\end{theorem}

For example, there exists an $\aleph_1$-existentially closed group of cardinality $2^{\aleph_2}$ having automorphism group of size $2^{\aleph_2}$.   It is easy to see that a $\kappa$-complete group $G$ has trivial center, implying that the conjugation map gives a monomorphism $G \rightarrow \Aut(G)$, so $|\Aut(G)| \geq |G|$.  Thus the number of automorphisms in Theorem \ref{powercardinal} is the minimal possible.  We give some lead-up lemmas in Section \ref{backgroundlemmas} and prove Theorem \ref{powercardinal} in Section \ref{Theproof}.  In Section \ref{Sizeofclosed} we show that for uncountable $\kappa$ there is a $\kappa$-existentially closed group of infinite size $\mu$ if and only if $\mu^{<\kappa} = \mu$ (Corollary \ref{biconditional}).

\end{section}

\begin{section}{Background lemmas}\label{backgroundlemmas}

We provide some preliminary results to prepare for the proof of Theorem \ref{powercardinal}.

\begin{notation}
We shall mostly follow the set theoretic conventions in \cite{J}.  Letters $\alpha$, $\beta$, $\gamma$, $\delta$, $\epsilon$, $\zeta$, $\eta$ will be used for ordinals and $\theta$, $\kappa$, $\lambda$, $\mu$ for infinite cardinals.  Each ordinal is the set of ordinals which are strictly below itself, and a cardinal is an ordinal which does not biject with a strictly smaller ordinal.  We write $\alpha + 1$ for the successor ordinal to $\alpha$ and $\kappa^+$ for the successor cardinal to $\kappa$.  The expression $|X|$ denotes the cardinality of $X$ (the unique cardinal with which $X$ bijects) and $[X]^{\mathfrak{n}}$ is the collection of subsets of $X$ of cardinality $\mathfrak{n}$.  Cardinal exponentiation is written in the standard way $\kappa^{\lambda}$ and $\kappa^{< \lambda} = \sup\{\kappa^{\theta}: \theta \text{ is a cardinal and }\theta < \lambda\}$.  We will use parenthetical notation $(X_i : i \in I)$ to exhibit an indexed sequence (the angular $\langle  \rangle$ will be used to denote subgroup generation).  If $f: X \rightarrow Y$ and $Z \subseteq X$ then $f''Z = \{f(x): x \in Z\}$.  We 

We let $\cf(\alpha)$ denote the \emph{cofinality} of $\alpha$ (i.e. the minimal size of an unbounded subset of $\alpha$).  An infinite cardinal $\kappa$ is \emph{regular} if $\cf(\kappa) = \kappa$, and is \emph{singular} otherwise.  Recall that each successor cardinal $\kappa^+$ is regular.  A subset $C$ of a cardinal $\kappa$ is \emph{club} if it is closed (in the order topology on $\kappa$) and unbounded in $\kappa$.  Recall that if $\kappa$ has uncountable cofinality then the intersection of two clubs in $\kappa$ is again club.  A subset $S$ of an infinite cardinal $\kappa$ is \emph{stationary} if it has nonempty intersection with each club in $\kappa$.

A sequence $(X_{\alpha}: \alpha < \gamma)$ which is $\subseteq$-increasing (i.e. $X_{\alpha_0} \subseteq X_{\alpha_1}$ if $\alpha_0 \leq \alpha_1$) is \emph{continuous} if $X_{\beta} = \bigcup_{\alpha < \beta} X_{\alpha}$ for $\beta < \gamma$ a non-zero limit ordinal.  Recall that a collection $\mathcal{J}$ of subsets of $\kappa$ is \emph{independent} if for distinct $X_1$, $\ldots$, $X_n$, $Y_1$, $\ldots$, $Y_m \in \mathcal{J}$ the intersection $$X_1 \cap \cdots \cap X_n \cap (\kappa \setminus Y_1) \cap \cdots \cap (\kappa \setminus Y_m)$$ has cardinality $\kappa$.
\end{notation}

\begin{lemma}\label{independentset}\cite[Lemma 7.7]{J}
For each infinite cardinal $\kappa$ there exists an independent collection $\mathcal{J}$ of subsets such that $|\mathcal{J}| = 2^{\kappa}$.
\end{lemma}

The following result is  due to the second author.  Note that when we invoke this lemma, we'll have $(C)_2$ hold for quite trivial reasons, and therefore we do not bother to define what Jonsson cardinals are.

\begin{lemma}\label{beautifulcoloringexhaustion} \cite[Theorem 3.5(3)]{She1}  If $\lambda \geq \aleph_1$ there exists $c: [\lambda]^2 \rightarrow \lambda$ such that in the following listed statements, the conjunction $(C)_1 \wedge (C)_2 \wedge (C)_3$ implies $(C)_4$.

\begin{itemize}

\item[$(C)_1$] $\kappa = \cf(\kappa) \neq \cf(\lambda)$.

\item[$(C)_2$] If $\lambda = \chi^+$ with $\chi$ singular and $\chi = \sup\{\theta < \delta : \theta \text{ is a regular Jonsson cardinal}\}$ then $\kappa \neq \cf(\chi)$.

\item[$(C)_3$] $(U_ {\epsilon}: \epsilon \leq \kappa)$ is a continuous $\subseteq$-increasing sequence whose union is $\lambda$.

\item[$(C)_4$]$c''[U_{\epsilon}]^2 = \lambda$ for some $\epsilon < \kappa$.

\end{itemize}

\end{lemma}

\begin{notation}
Our notation for group theory will be fairly standard.  If $G$ is a group and $X \subseteq G$, $\langle X \rangle$ is the subgroup of $G$ generated by $X$ and $\langle\langle X \rangle\rangle$ is the smallest normal subgroup in $G$ including the set $X$.  We'll use the convention that $\equiv$ will mean that two words are equivalent, reading letter-by-letter, and write equality $=$ if two expressions represent the same element in the group.
\end{notation}

In the constructions that follow we'll use small cancellation over free products (see \cite[Chapter V]{LS}) and provide now a review of the relevant concepts.  Recall that each element of a free product $H = *_{i \in I} H_i$ of groups is represented uniquely as a \emph{reduced word}, i.e. as a word $W \equiv g_1g_1 \cdots g_n$ where $g_s \in (\bigcup_{i \in I} H_i) \setminus \{1\}$ and $g_s \in H_{i_s}$ implies $g_{s + 1} \notin H_{i_s}$ for each $1 \leq s < n$.  We write $\Len(W) = n$ for the length of this word.

If $W_1 \equiv g_1 \cdots g_n$ and $W_2 \equiv h_1 \cdots h_k$ are reduced words then one obtains the reduced word representative of their product in $H$ by first letting $r$ be maximal such that $g_{k + 1 - t} = h_t^{-1}$ for all $1 \leq t \leq r$.  The reduced word representative is $g_1\cdots g_{n - r - 1} h h_{r + 2} \cdots h_k$ in case $g_{n - r}, h_{r + 1} \in H_i$ and $g_{n - r}h_{r + 1} = h$.  Otherwise the reduced word representative is $g_1 \cdots g_{n - r}h_{r + 1}\cdots h_k$.  If word $W$ is the concatenation $W \equiv W_1W_2$, with $W_1$ and $W_2$ reduced, and the value $r$ used in computing the reduced representative for $W$ is zero, then we say $W$ is \emph{semi-reduced}.

A reduced word $W \equiv g_1 \cdots g_n$ is \emph{cyclically reduced} if either $\Len(W) \leq 1$ or $g_1$ and $g_n$ are in distinct $H_i$.  More generally reduced word $W$ is \emph{weakly cyclically reduced} if either $\Len(W) \leq 1$ or $g_1 \neq g_n^{-1}$.  We'll say a subset $\mathcal{R} \subseteq H$ is \emph{symmetrized} if each element $W \in \mathcal{R}$ is weakly cyclically reduced and every weakly cyclically reduced conjugate of $W$ and of $W^{-1}$ is also in $\mathcal{R}$.  If $\mathcal{R}$ is symmetrized we say a word $U$ is a \emph{piece} if there exist distinct $W_1, W_2 \in \mathcal{R}$ with semi-reduced forms $W_1 = UV_1$ and $W_2 = UV_2$.

\begin{definition}  \label{smallcanc}  A symmetrized set $\mathcal{R} \subseteq H = *_{i \in I} H_i$ \emph{satisfies the $C'(\frac{1}{n})$ condition} if for each $W \in \mathcal{R}$ we have

\begin{enumerate}

\item $\Len(W) > n$; and

\item whenever $W = UV$ is a semi-reduced form for $W$, with $U$ a piece, we have $n\Len(U) < L(W)$.

\end{enumerate}

\end{definition}

The following is an immediate consequence of \cite[Theorem V.9.5]{LS}.

\begin{lemma}\label{LyndonSchupp}  Let $H = *_{i\in I} H_i$ be a free product and $\mathcal{R}$ a symmetrized subset of $H$ which satisfies $C'(\frac{1}{6})$.  Then the natural map $q: H \rightarrow H/\langle\langle \mathcal{R}\rangle\rangle$ is injective on $H_i$ for each $i \in I$, and $q(H_i) \cap q(H_j) = \{1\}$ when $i \neq j$.
\end{lemma}

\begin{lemma}\label{intoquickgeneration}  Suppose $L$ is a group, $(g_{\alpha}: \alpha < \lambda)$ is an enumeration of the elements of $L \setminus \{1\}$, and $f: [\lambda]^2 \rightarrow \lambda$ is a surjective function.  Then $L$ can be embedded in a group $\overline{L_0}$, with $\overline{L_0}$ having four subgroups $H_{p, \ell}$, $p, \ell \in \{1, 2\}$ such that

\begin{enumerate}[(a)]

\item for $p, \ell \in \{1, 2\}$ the group $H_{p, \ell}$ is the direct sum $\bigoplus_{\alpha< \lambda} \langle x_{\alpha, p, \ell}\rangle$ with each $x_{\alpha, p, \ell}$ of order two;

\item $\overline{L_0}$ is generated by the union $H_{1, 1} \cup H_{1, 2} \cup H_{2, 1} \cup H_{2, 2}$;

\item $g_{f\{\alpha, \beta\}} \in \langle \{x_{\alpha, 1, 1}, x_{\alpha, 2, 1}, x_{\beta, 1, 2}x_{\beta, 2,  2}\} \rangle$;

\item $H_{p, \ell} \cap H_{p', \ell'} = \{1\}$ when $(p, \ell) \neq (p', \ell')$.

\end{enumerate}

\end{lemma}

\begin{proof}
Assume the hypotheses.  In accordance with condition (a) take $H_{p, \ell}$ to be the direct sum of cyclic groups $\langle x_{\alpha, p, \ell}\rangle$ of order two.  Let $H_{\ell}$ be the free product $H_{1, \ell} * H_{2, \ell}$, for $\ell  = 1, 2$.  Consider the free product $H = L * H_1 * H_2$.  Let $c_{\alpha, \ell}$ equal the commutator $[x_{\alpha, 1, \ell}, x_{\alpha, 2, \ell}] \in H_{\ell}$.  For each $\alpha < \beta < \lambda$ take $W_{\alpha, \beta}$ to be the word

$$g_{f\{\alpha, \beta\}}^{-1} c_{\alpha, 1}^{12}c_{\beta, 2}c_{\alpha, 1}^{11}c_{\beta, 2} \cdots c_{\alpha, 1}^{2}c_{\beta, 2}c_{\alpha, 1}c_{\beta, 2}$$

\noindent Let $\mathcal{R}$ be the symmetrized closure of the set $\{W_{\alpha, \beta}\}_{\alpha < \beta < \lambda}$ (i.e. the set of all weakly cyclically reduced conjugates of elements in $\{W_{\alpha, \beta}\}_{\alpha < \beta < \lambda}$ and their inverses).  It is easy to see in each case that $\Len(W_{\alpha, \beta}) = 25$ in the free product $H$ and $W_{\alpha, \beta}$ is cyclically reduced, and so each element of $\mathcal{R}$ has length $25$ or $26$.

For any $W \in \mathcal{R}$ the initial subword of length $3$ uniquely elucidates $\alpha < \beta < \lambda$ and $q \in \{\pm 1\}$ for which $W$ is conjugate to $W_{\alpha, \beta}^q$.  Also, all weakly cyclically reduced conjugates of $W_{\alpha, \beta}^q$ are uniquely determined by their initial subword of length $3$, so a piece $U$ in $\mathcal{R}$ can have length at most $4$.  Thus $6 \cdot \Len(U) < 25 \leq \Len(W)$ for any piece $U$ of any $W \in \mathcal{R}$, and $\Len(W) > 6$.  Therefore $\mathcal{R}$ satisfies the $C'(\frac{1}{6})$ condition and by Lemma \ref{LyndonSchupp} the quotient $\overline{L_0} = H/\langle\langle\mathcal{R}\rangle\rangle$ embeds each of $L$, $H_1$, and $H_2$ with trivial pairwise intersection.  Thus condition (d) holds, condition (c) is clear by the relators $W_{\alpha, \beta}$, and condition (b) holds since $f$ is onto.
\end{proof}

\begin{lemma}\label{morecontrol}  Suppose $L$ is a group, $\lambda = 2^{\mu}$, $(g_{\alpha}: \alpha < \lambda)$ is an enumeration of $L \setminus \{1\}$ and $f: [\lambda]^2 \rightarrow \lambda$ is surjective.  Then $L$ can be embedded in a group $\overline{L_{00}}$ having four subgroups $H_{p, \ell}$, $p, \ell \in \{1, 2\}$ and a subset $\{z_{\gamma}: \gamma < \mu\} \subseteq \overline{L_{00}}$ such that

\begin{enumerate}[(a)]

\item for $p, \ell \in \{1, 2\}$ the group $H_{p, \ell}$ is the direct sum $\bigoplus_{\alpha< \lambda} \langle x_{\alpha, p, \ell}\rangle$ with $x_{\alpha, p, \ell}$ of order two;

\item $\overline{L_{00}}$ is generated by the union $H_{1, 1} \cup H_{1, 2} \cup H_{2, 1} \cup H_{2, 2} \cup \{z_{\gamma}: \gamma < \mu\}$;

\item $g_{f\{\alpha, \beta\}} \in \langle \{x_{\alpha, 1, 1}, x_{\alpha, 2, 1}, x_{\beta, 1, 2}x_{\beta, 2,  2}\} \rangle$;

\item $H_{p, \ell} \cap H_{p', \ell'} = \{1\}$ when $(p, \ell) \neq (p', \ell')$;

\item if $p, \ell \in \{0, 1\}$ and $\alpha_1, \ldots, \alpha_m, \hat{\alpha_1}, \ldots,  \hat{\alpha_n}$ are distinct elements of $\lambda$, then there exist $\gamma, \hat{\gamma} < \mu$ such that $z_{\gamma}$ commutes with $x_{\alpha_r, p, \ell}$ for $1 \leq r \leq m$ but not with any $x_{\hat{\alpha_s}, p, \ell}$ for $1 \leq s \leq n$ and $z_{\hat{\gamma}}$ commutes with $x_{\hat{\alpha_s}, p, \ell}$ for $1 \leq s \leq n$ but does not commute with $x_{\alpha_r, p, \ell}$ for $1 \leq r \leq m$.

\end{enumerate}
\end{lemma}

\begin{proof}
Take $\overline{L_0}$ and $H_{p, \ell}$ given by Lemma \ref{intoquickgeneration}.  By Lemma \ref{independentset} let $\mathcal{J}$ be an independent collection of subsets of $\mu$ with $|\mathcal{J}| = 2^{\mu}$.  Let $b: \mathcal{J} \rightarrow \lambda$ to be a bijection.  For each $\gamma < \mu$ let $Y_{\gamma} = \{\alpha < \lambda: \gamma \in b^{-1}(\alpha)\}$ and  let $\overline{L_{00}}$ be the HNN extension over $\overline{L_0}$ which adds a stable letter $t_{\gamma, p, \ell}$ for each $(\gamma, p, \ell) \in \mu \times \{1, 2\} \times \{1, 2\}$ so that $$t_{\gamma, p, \ell}^{-1}gt_{\gamma, p, \ell} = g$$ for $g\in \langle \{x_{\alpha, p, \ell}\}_{\alpha \in Y_{\gamma}}\rangle$ for each $\gamma, p, \ell$ .

By \cite[Theorem IV.2.1]{LS} it is clear that in fact $t_{\gamma, p, \ell}$ commutes with an element $g$ of $\overline{L_0}$ if and only if $g \in \langle \{x_{\alpha, p, \ell}\}_{\alpha \in Y_{\gamma}} \rangle$.  Taking a bijection $b': \mu \rightarrow \mu \times \{1, 2\} \times \{1, 2\}$ we let $z_{\gamma} = t_{b'(\gamma)}$.  Conditions (a) - (d) evidently hold.

Suppose $p, \ell \in \{1, 2\}$ and $\alpha_1, \ldots, \alpha_m, \hat{\alpha_1}, \ldots, \hat{\alpha_n}$ are distinct elements of $\lambda$ with $n \geq 2$.  Pick $\gamma' \in \bigcap_{r = 1}^m b^{-1}(\alpha_r) \setminus (\bigcup_{s = 1}^n b^{-1}(\hat{\alpha_s}))$.  Notice that $t_{\gamma', p, \ell}$ commutes with $x_{\alpha_r, p, \ell}$ as $\gamma' \in Y_{\alpha_r}$, but $t_{\gamma', p_1, \ell_1}$ does not commute with $x_{\hat{\alpha_s}, p, \ell}$ for any $1\leq s \leq n$ since $\gamma' \notin Y_{\hat{\alpha_s}}$.  Selecting $\gamma'' \in \bigcap_{s = 1}^n b^{-1}(\hat{\alpha_s}) \setminus (\bigcup_{r = 1}^m b^{-1}(\alpha_r))$ we similarly see that $t_{\gamma'', p, \ell}$ commutes with each $x_{\hat{\alpha_s}, p, \ell}$, $1 \leq s \leq n$, but does not commute with $x_{\alpha_r, p, \ell}$ for $1 \leq r \leq m$.  Now $\gamma = (b')^{-1}(\gamma', p, \ell)$ and $\hat{\gamma} = (b')^{-1}(\gamma'', p, \ell)$ satisfy the conclusion of (e).

\end{proof}

We will also make use of the following.

\begin{lemma}\label{intoaclosedgroup} \cite[p. 227]{KayaKegelKuzucuoglu}  Suppose $\aleph_1 \leq \kappa$, $\lambda = \lambda^{< \kappa}$ and $L$ is a group of cardinality at most $\lambda$.  Then there exists a $\kappa$-existentially closed group $\overline{L}$ with $L \leq \overline{L}$ and $|\overline{L}| = \lambda$.
\end{lemma}

\end{section}

\begin{section}{Proof of Theorem \ref{powercardinal}} \label{Theproof}

We are now ready to prove Theorem \ref{powercardinal}, and so assume its hypotheses.  In case $\lambda$ is the successor of a singular cardinal $\chi$, select $\overline{\kappa} \in \{\kappa, \kappa^+\} \setminus \{\cf(\chi)\}$, otherwise let $\overline{\kappa} = \kappa$.  By our choice of $\overline{\kappa}$ choose by Lemma \ref{beautifulcoloringexhaustion} a function $c: [\lambda]^2 \rightarrow \lambda$ such that if $(U_i: i < \overline{\kappa})$ is a continuous $\subseteq$-increasing sequence whose union is $\lambda$ there exists $i < \overline{\kappa}$ for which $c''[U_i]^2 = \lambda$.   Let $(S_{\epsilon}: \epsilon <\overline{\kappa})$ be a collection of pairwise disjoint stationary subsets of $\overline{\kappa}$ \cite[Theorem 8.10]{J}, where without loss of generality $\min (S_{\epsilon}) > \epsilon$.

We inductively construct a continuous $\leq$-increasing sequence of groups $(G_{\epsilon}: \epsilon \leq \overline{\kappa})$ having $|G_{\epsilon}| = \lambda$ for each $\epsilon$.  As the sequence is of length $\overline{\kappa} < \lambda$ it will be clear for limit ordinal $\epsilon$ that $|G_{\epsilon}|$ has cardinality $\lambda$ provided this holds at the earlier stages.  Take $G_0$ to be an arbitrary group of cardinality $\lambda$.  Suppose that we have defined $G_{\zeta}$ and $\zeta < \overline{\kappa}$.

\begin{itemize}

\item[$(*)_1$]  If $\zeta = 0 \modvar 3$ then by Lemma \ref{intoaclosedgroup} let $G_{\zeta + 1} \geq G_{\zeta}$ be a $\kappa$-existentially closed group of cardinality $\lambda$.

\item[$(*)_2$]  If $\zeta = 1 \modvar 3$ then let $(g_{\zeta, \alpha}: \alpha < \lambda)$ list the elements of $G_{\zeta} \setminus \{1\}$.  By Lemma \ref{morecontrol} take $G_{\zeta + 1} \geq G_{\zeta}$ to have subsets $\{x_{\zeta + 1, \alpha, p, \ell}: \alpha < \lambda, p, \ell \in \{1, 2\}\}$ and $\{z_{\zeta + 1, \gamma}: \gamma < \mu\}$ such that

\begin{enumerate}[(i)]

\item $G_{\zeta + 1}$ is generated by the union $$\{x_{\zeta + 1, \alpha, p, \ell}: \alpha < \lambda, p, \ell \in \{1, 2\}\} \cup \{z_{\zeta + 1, \gamma}: \gamma < \mu\};$$

\item $g_{\zeta, c\{\alpha, \beta\}} \in \langle \{x_{\zeta + 1, \alpha, 1, 1}, x_{\zeta + 1, \alpha, 2, 1}, x_{\zeta + 1, \beta, 1, 2}, x_{\zeta + 1, \beta, 2, 2}\} \rangle$;

\item if $p, \ell \in \{0, 1\}$ and $\alpha_1, \ldots, \alpha_m, \hat{\alpha_1}, \ldots,  \hat{\alpha_n}$ are distinct elements of $\lambda$, then there exist $\gamma, \hat{\gamma} < \mu$ such that $z_{\zeta + 1, \gamma}$ commutes with $x_{\zeta + 1, \alpha_r, p, \ell}$ for $1 \leq r \leq m$ but not with any $x_{\zeta + 1, \hat{\alpha_s}, p, \ell}$ for $1 \leq s \leq n$ and $z_{\hat{\gamma}}$ commutes with $x_{\zeta + 1, \hat{\alpha_s}, p, \ell}$ for $1 \leq s \leq n$ but does not commute with $x_{\zeta + 1, \alpha_r, p, \ell}$ for $1 \leq r \leq m$; and

\item if $p, \ell \in \{1, 2\}$ the subgroup $H_{\zeta + 1, p, \ell} = \langle \{x_{\zeta + 1, \alpha, p, \ell}: \alpha < \lambda\} \rangle$ is the direct sum of the cyclic subgroups of the listed elements, each having order two.

\end{enumerate}

\item[$(*)_3$] If $\zeta = 2 \modvar 3$, $\zeta =  \delta + 2$ with $\delta \in S_{\epsilon_1}$ and $\epsilon_1 = 2 \modvar 3$ then take $G_{\zeta + 1} \geq G_{\zeta}$ with $|G_{\zeta + 1}| = \lambda$ to have elements $y_{\zeta + 1, p, \ell}$ such that $g \in G_{\delta}$ commutes with $y_{\zeta + 1, p, \ell}$ if and only if $g \in H_{\epsilon_1, p, \ell}$.  To obtain this, one can let $\iota_{p, \ell}$ be the identity homomorphism on $H_{\epsilon_1, p, \ell}$ and take $G_{\zeta + 1}$ to be the HNN extension of $G_{\zeta}$ over the maps $\iota_{p, \ell}$ with $y_{\zeta + 1, p, \ell}$ as the stable letters.  In case $\zeta - 2$ is not in such a specified stationary set we let $G_{\zeta + 1} \geq G_{\zeta}$ be $\kappa$-existentially closed of cardinality $\lambda$.
\end{itemize}

We check that the group $G = \bigcup_{\epsilon < \overline{\kappa}} G_{\epsilon}$ is as required.

\begin{claim}\label{club}  If $\pi \in \Aut(G)$ the set $C_{\pi} = \{\delta < \overline{\kappa} : \pi \upharpoonright G_{\delta} \in \Aut(G_{\delta})\}$ is club in $\overline{\kappa}$.
\end{claim}

\begin{proof}
It is clear that the set $C_{\pi}$ is closed.  Note that if $\zeta < \overline{\kappa}$ there is an $\zeta < \eta < \overline{\kappa}$ such that $\pi(G_{\zeta}) \subseteq G_{\eta}$.  To see this, select $\zeta_1 > \zeta$ with $\zeta_1 = 1 \modvar 3$.  For $i < \overline{\kappa}$ let $U_i = \{\alpha < \lambda: \pi(x_{\zeta_1, \alpha, p, \ell}) \in G_i, p, \ell \in \{1, 2\}\}$.  Now $(U_i: i < \overline{\kappa})$ is a $\subseteq$-increasing continuous sequence whose union is $\lambda$ and therefore $c''[U_{\eta}]^2 = \lambda$ for some $\eta < \overline{\kappa}$.  By $(*)_2$ condition (ii) we know $\pi(G_{\zeta}) \subseteq \pi(G_{\zeta_1 - 1}) \subseteq G_{\eta}$.  Now to see that $C_{\pi}$ is unbounded, let $\alpha_0 < \overline{\kappa}$ be given and generally pick $\alpha_{n + 1} > \alpha_n$ such that $\pi(G_{\alpha_n}) \subseteq G_{\alpha_{n + 1}}$.  Evidently  $\sup\{\alpha_n : n \in \omega\} \in C_{\pi}$.
\end{proof}

\begin{claim}\label{agreement}  If $\pi_1, \pi_2 \in \Aut(G)$ agree on the set $$\{y_{\delta + 3, p, \ell}: \delta < \overline{\kappa}, \delta = 0\modvar 3, p, \ell \in \{1, 2\}\} \cup \{z_{\delta + 2, \gamma}: \delta < \overline{\kappa}, \delta = 0 \modvar 3, \gamma < \mu\}$$ then $\pi_1 = \pi_2$.

\end{claim}

\begin{proof}
Let $\epsilon < \overline{\kappa}$ be given.   Select $\epsilon < \epsilon_1 < \overline{\kappa}$ with $\epsilon_1 = 2 \modvar 3$ .  Take a limit ordinal $\delta \in C_{\pi_1} \cap C_{\pi_2} \cap S_{\epsilon_1}$.  We know $\pi_1(G_{\delta}) = \pi_2(G_{\delta}) = G_{\delta}$.  By $(*)_3$ the element $\pi_1(y_{\delta + 3, p, \ell})$ commutes with  $g \in G_{\delta}$ if and only if $g \in \pi_1(H_{\epsilon_1, p, \ell})$, and similarly $\pi_2(y_{\delta + 3, p, \ell})$ commutes with  $g \in G_{\delta}$ if and only if $g \in  \pi_2(H_{\epsilon_1, p, \ell})$.  Since $\pi_1(y_{\delta + 3, p, \ell}) = \pi_2(y_{\delta + 3, p, \ell})$ we know $\pi_1(H_{\epsilon_1, p, \ell}) =\pi_2(H_{\epsilon_1, p, \ell})$.

For $g \in H_{\epsilon_1, p, \ell}$ let $J(g) = \{\beta < \mu: [z_{\epsilon_1, \beta}, g] = 1\}$.  By $(*)_2$ (ii) we know $H_{\epsilon_1, p, \ell} = \bigoplus_{\alpha < \lambda} \langle x_{\epsilon_1, \alpha, p, \ell} \rangle$ so we can write $g \in H_{\epsilon_1, p, \ell}$ uniquely as a repetition-free sum $$g = \sum_{r = 1}^m x_{\epsilon_1, \alpha_r, p, \ell}.$$  Then $J(g) = \bigcap_{r = 1}^n J(x_{\epsilon_1, \alpha_r, p, \ell})$  where the intersection is understood to be $\mu$ if the index is empty.  Therefore for distinct $g = \sum_{r = 1}^m x_{\epsilon_1, \alpha_r, p, \ell}$ and $\hat{g} = \sum_{s = 1}^n x_{\epsilon_1, \hat{\alpha_s}, p, \ell}$ we have $J(g) \neq J(\hat{g})$ by $(*)_2$ (iii).  Then for $g \in H_{\epsilon_1, p, \ell}$ we see that

$$
\begin{array}{ll}
\{\pi_1(g)\} &  = \bigcap_{\gamma \in J(g)} \{h \in \pi_1(H_{\epsilon_1, p, \ell}): [h, \pi_1(z_{\epsilon_1, \gamma})]= 1\}\vspace*{2mm}\\
& \cap \bigcap_{\gamma \in \mu \setminus J(g)} \{h \in \pi_1(H_{\epsilon_1, p, \ell}): [h, \pi_1(z_{\epsilon_1, \gamma})]\neq 1\} \vspace*{2mm}\\
& = \bigcap_{\gamma \in J(g)} \{h \in \pi_2(H_{\epsilon_1, p, \ell}): [h, \pi_2(z_{\epsilon_1, \gamma})]= 1\}\vspace*{2mm}\\
& \cap \bigcap_{\gamma \in \mu \setminus J(g)} \{h \in \pi_2(H_{\epsilon_1, p, \ell}): [h, \pi_2(z_{\epsilon_1, \gamma})]\neq 1\} \vspace*{2mm}\\
& = \{\pi_2(g)\}. \vspace*{2mm}
\end{array}
$$

Then $\pi_1$ and $\pi_2$ agree on $H_{\epsilon_1, p, \ell}$, and since $$G_{\epsilon} \leq G_{\epsilon_1 - 1} \leq \langle \bigcup_{p, \ell \in \{1, 2\}} H_{p, \ell} \rangle$$ we know $\pi_1$ and $\pi_2$ agree on $G_{\epsilon}$, and as $\epsilon$ was arbitrary we know $\pi_1 = \pi_2$ globally.
\end{proof}

Now we can conclude that $G$ has the desired properties for Theorem \ref{powercardinal}.  Of course, $|G| = \lambda$.  We know $G$ is $\kappa$-existentially closed as it is a union, of cofinal length at least $\kappa$, of $\kappa$-existentially closed groups.  Finally, $\lambda \leq |\Aut(G)| \leq |G|^{\mu} = (2^{\mu})^{\mu} = \lambda$ by Claim \ref{agreement}.

\end{section}

\begin{section}{Requirement on the size of a $\kappa$-existentially closed group}\label{Sizeofclosed}

We prove the neccesity of the condition on cardinality in Lemma \ref{intoaclosedgroup}.  We will make use of the following from the existing literature.

\begin{lemma}\label{poweroftwo}\cite[Proposition 3]{KayaKegelKuzucuoglu}  If there exists a $\kappa$-existentially closed group of cardinality $\lambda$ then $2^{<\kappa} \leq \lambda$.
\end{lemma}

\begin{theorem} \label{sizeofexclosed}  Suppose that a group $G$ is $\kappa$-existentially closed, with $\kappa$ uncountable and $|G| = \mu$.  Then $\mu^{< \kappa} = \mu$.
\end{theorem}

\begin{proof}
Assume the hypotheses.  By Lemma \ref{poweroftwo} we know that $|G| \geq 2^{< \kappa} \geq \kappa$.  If $\mu \leq 2^{< \kappa}$ then indeed $\mu^{< \kappa} = (2^{< \kappa})^{< \kappa} = 2^{< \kappa} = \mu$, so we are done in this case.  Therefore we may assume that $\mu > 2^{< \kappa}$.  Toward contradiction suppose that $\mu^{<\kappa} > \mu$.  Were it the case that $\mu^{\theta} \leq \mu$ for all $\theta < \kappa$, then as $\mu \geq \kappa$ we would have $\mu = \mu \cdot \mu \geq \sum_{\theta < \kappa} \mu^{\theta} = \mu^{<\kappa}$ contrary to our assumption.  Therefore there is some $\theta < \kappa$ such that $\mu^{\theta} > \mu$, and let $\sigma$ be minimal for which this is true.  Note that $\sigma$ is a regular cardinal, for if $\sigma = \sum_{\alpha < \cf(\sigma)} \lambda_{\alpha}$ with $\cf(\sigma), \lambda_{\alpha} < \sigma$ then $\mu^{\sigma} = \mu^{\sum_{\alpha < \cf(\sigma)} \lambda_{\alpha}} = \prod_{\alpha < \cf(\sigma)} \mu^{\lambda_{\alpha}} = \mu^{\cf(\sigma)} = \mu$, a contradiction.  Furthermore let $\chi = \min\{\lambda \mid \lambda^{\sigma} > \mu\}$, so of course $\chi \leq \mu$.  Also, clearly $\chi > 2^{<\kappa} >  \sigma$ as $\chi^{\sigma} > \mu \geq 2^{<\kappa}$.

We claim that $\cf(\chi) = \sigma$.  If on one hand $\cf(\chi) > \sigma$ then the image of any function from $\sigma$ to $\chi$ will be bounded (as $\sigma$ is regular), and so $\chi^{\sigma} = |\bigcup_{\delta < \chi} \delta^{\sigma}| \leq \chi \cdot \mu = \mu$, contradicting the choice of $\chi$.  On the other hand if $\cf(\chi) < \sigma$ then $\cf(\chi) < \sigma < \chi$ and we express $\chi = \sum_{\alpha < \cf(\chi)} \lambda_{\alpha}$, where $\lambda_{\alpha} < \chi$, so that $$\chi^{\sigma} \leq \cf(\chi)^{\sigma}\cdot \prod_{\alpha < \cf(\chi)}\lambda_{\alpha}^{\sigma} \leq \mu \cdot \mu^{\cf(\chi)} = \mu$$ is a contradiction in this case as well.  Thus indeed $\cf(\chi) = \sigma$.

Let $\{\chi_{\alpha}\}_{\alpha < \sigma}$ be an increasing sequence of cardinals below $\chi$ such that $\sum_{\alpha < \sigma} \chi_{\alpha} = \chi$ and $\chi_0 > 2^{< \kappa}$.  We construct collections $\{\{g_{\nu}\}_{\nu \in \prod_{\alpha \leq \beta} \chi_{\alpha}}\}_{\beta < \sigma}$ of elements of $G$ such that for each $\beta < \sigma$ we have

\begin{enumerate}

\item[$(\diamondsuit)_1$] if $\nu \in \prod_{\alpha \leq \beta} \chi_{\alpha}$ and $\alpha_0 < \alpha_1 < \beta$ we have $g_{\nu \upharpoonright \alpha_0} \neq g_{\nu \upharpoonright \alpha_1}$;

\item[$(\diamondsuit)_2$] if $\eta, \nu \in \prod_{\alpha \leq \beta} \chi_{\alpha}$ with $\eta(\alpha) = \nu(\alpha)$ for all $\alpha < \beta$ and $\eta(\beta) \neq \nu(\beta)$ then $g_{\eta} \neq g_{\nu}$;

\item[$(\diamondsuit)_3$] if $\eta, \nu \in \prod_{\alpha \leq \beta} \chi_{\alpha}$ then the map given by $g_{\eta\upharpoonright\alpha} \mapsto g_{\nu\upharpoonright\alpha}$ for all $\alpha \leq \beta$ induces an isomorphism from $\langle \{g_{\eta\upharpoonright\alpha}\}_{\alpha \leq \beta} \rangle$ to $\langle \{g_{\nu\upharpoonright\alpha)}\}_{\alpha \leq \beta} \rangle$.

\end{enumerate}

For $\beta = 0$ we take $\{g_{\gamma}\}_{\gamma < \chi_{0}} \subseteq G$ to be a set of distinct elements which are all of the same order (it is possible to find such a subset because $\chi_0 < \mu$).  Clearly $(\diamondsuit)_1$ -- $(\diamondsuit)_3$ are satisfied.  Assume that we have selected $\{g_{\nu}\}_{\nu \in \prod_{\alpha < \delta}}$ for all $\delta < \beta < \sigma$ satisfying $(\diamondsuit)_1$ -- $(\diamondsuit)_3$.  Note that for each $\nu, \eta \in \prod_{\alpha < \beta} \chi_{\alpha}$ the map $g_{\nu \upharpoonright \alpha} \mapsto g_{\eta \upharpoonright \alpha}$ gives an isomorphism from the subgroup $\langle \{g_{\nu \upharpoonright \alpha}\}_{\alpha < \beta} \rangle$ to $\langle \{g_{\eta \upharpoonright \alpha}\}_{\alpha < \beta} \rangle$ (by $(\diamondsuit)_3$).  Let $\zeta \in \prod_{\alpha < \beta} \chi_{\alpha}$ be constantly $0$, and for each $g \in G \setminus \langle \{g_{\zeta \upharpoonright \alpha}\}_{\alpha < \beta} \rangle$ consider the homomorphism $\phi_g$ from the free product $\langle \{g_{\zeta \upharpoonright \alpha}\}_{\alpha < \beta} \rangle * \langle x \rangle_{\infty}$ to $G$ which is identity on $\langle \{g_{\zeta \upharpoonright \alpha}\}_{\alpha < \beta} \rangle$ and with $x \mapsto g$.  Since $|\beta| < \mu$ the number of such homomorphisms $\phi_g$ is $\mu$, but the number of kernels $\ker(\phi_g)$ is at most $2^{|\beta|\cdot\aleph_0} \leq 2^{\sigma} < \mu$, so we may select a subset $\{g_{\zeta \frown \gamma}\}_{\gamma < \chi_{\beta}}$ of distinct elements of $G \setminus \langle \{g_{\zeta \upharpoonright \alpha}\}_{\alpha < \beta} \rangle$ such that $\ker(\phi_{g_{\zeta \frown \gamma}}) = \ker(\phi_{g_{\zeta \frown \gamma'}})$ for $\gamma, \gamma' < \chi_{\beta}$.  Clearly the map given by $g_{\zeta \upharpoonright \alpha} \mapsto g_{\zeta \upharpoonright \alpha}$ and $g_{\zeta \frown \gamma} \mapsto g_{\zeta \frown \gamma'}$ gives an isomorphism from $\langle \{g_{\zeta \upharpoonright \alpha}\}_{\alpha < \beta} \cup \{g_{\zeta \frown \gamma}\}\rangle$ to $\langle \{g_{\zeta \upharpoonright \alpha}\}_{\alpha < \beta} \cup \{g_{\zeta \frown \gamma'}\}\rangle$.  As $G$ is $\kappa$-existentially closed, for each $\eta \in \prod_{\alpha < \beta} \chi_{\alpha}$ select $h_{\eta} \in G$ such that $h_{\eta}^{-1} g_{\zeta \upharpoonright \alpha} h_{\eta} = g_{\eta \upharpoonright \alpha}$ for all $\alpha < \beta$ and define $g_{\eta \frown \gamma} = h_{\eta}^{-1} g_{\zeta \frown \gamma} h_{\eta}$ for each $\gamma < \chi_{\beta}$.  That $(\diamondsuit)_1$ -- $(\diamondsuit)_3$ are satisfied is clear.

Now that we have our collections $\{\{g_{\nu}\}_{\nu \in \prod_{\alpha \leq \beta} \chi_{\alpha}}\}_{\beta < \sigma}$, we note by $(\diamondsuit)_3$ that for $\nu, \eta \in \prod_{\alpha < \sigma} \chi_{\alpha}$ the function $g_{\nu \upharpoonright \alpha} \mapsto g_{\eta \upharpoonright \alpha}$ gives an isomorphism from the subgroup $\langle \{g_{\nu \upharpoonright \alpha}\}_{\alpha < \sigma} \rangle$ to $\langle \{g_{\eta \upharpoonright \alpha}\}_{\alpha < \sigma} \rangle$.  Let $\zeta \in \prod_{\alpha < \sigma} \chi_{\alpha}$ be constantly $0$ and since $G$ is $\kappa$-existentially complete select for each $\eta \in \prod_{\alpha < \sigma} \chi_{\alpha}$ an $h_{\eta} \in G$ such that $h_{\eta}^{-1} g_{\zeta \upharpoonright \alpha} h_{\eta} = g_{\eta \upharpoonright \alpha}$ for all $\alpha < \sigma$.  Note if $\eta, \nu \in \prod_{\alpha < \sigma} \chi_{\alpha}$ are distinct with $\beta < \sigma$ minimal for which $\eta(\beta) \neq \nu(\beta)$ we have by requirement $(\diamondsuit)_2$ that $g_{\eta\upharpoonright(\beta + 1)} \neq g_{\nu\upharpoonright(\beta + 1)}$ and therefore $h_{\eta}^{-1} g_{\zeta \upharpoonright (\beta + 1)} h_{\eta} = g_{\eta \upharpoonright (\beta + 1)} \neq g_{\nu \upharpoonright (\beta + 1)} = h_{\nu}^{-1} g_{\zeta \upharpoonright (\beta + 1)} h_{\nu}$.  Thus the function from $\prod_{\alpha < \sigma} \chi_{\alpha}$ to $G$ given by $\eta \mapsto h_{\eta}$ is injective.  This is a contradiction since $\prod_{\alpha < \sigma} \chi_{\alpha} = \chi^{\sigma}$.

To see this last equality, let $\sigma =  \bigsqcup_{\epsilon < \sigma} B_{\epsilon}$ with each $|B_{\epsilon}| = \sigma$ for each $\epsilon$.  Then each $B_{\epsilon}$ is unbounded in $\sigma$ and for each $\epsilon$ we have $\prod_{\alpha \in B_{\epsilon}} \chi_{\alpha} \geq \sum_{\alpha \in B_{\epsilon}} \chi_{\alpha} = \chi$.  Therefore $\prod_{\alpha < \sigma} \chi_{\alpha} = \prod_{\epsilon < \sigma} \prod_{\alpha \in B_{\epsilon}} \chi_{\alpha} \geq \chi^{\sigma}$, and that $\prod_{\alpha < \sigma} \chi_{\alpha} \leq \chi^{\sigma}$ is quite clear, so $\prod_{\alpha < \sigma} \chi_{\alpha} = \chi^{\sigma}$.

\end{proof}

By combining Theorem \ref{sizeofexclosed} with Lemma \ref{intoaclosedgroup} we obtain the following.

\begin{corollary}\label{biconditional}  If $\kappa$ is uncountable and $\mu$ is infinite, the following are equivalent.

\begin{enumerate}

\item[(A)]  There is a $\kappa$-existentially closed group of size $\mu$.

\item[(B)] $\mu^{<\kappa} = \mu$.

\end{enumerate}

\end{corollary}

\end{section}

\end{document}